\documentclass[a4paper]{article}

\usepackage{amsfonts, amsmath, amssymb, amsthm}
\usepackage{graphicx}
\usepackage{fullpage}
\usepackage{float}

\usepackage{tikz}
\usetikzlibrary{calc}

\newtheorem{lemma}{Lemma}[section]
\newtheorem{theorem}{Theorem}[section]
\newtheorem{corollary}{Corollary}[theorem]
\newtheorem{conjecture}{Conjecture}[section]

\newcommand{\Z}{\mathbb{Z}}
\newcommand{\N}{\mathbb{N}}

\newcommand{\abs}[1]{\left| #1 \right|}
\renewcommand{\dim}[1]{\textnormal{dim}\left( #1 \right)}

\begin{document}
	\title{Most Clicks Problem in \textit{Lights Out}}
	\author{William Boyles\thanks{Department of Mathematics, North Carolina State University, Raleigh, NC 27695 (wmboyle2@ncsu.edu)}}
	\date{\today}
	\maketitle
	
	\begin{abstract}
		Consider a game played on a simple graph $G = (V, E)$ where each vertex consists of a clickable light. Clicking any vertex $v$ toggles the on/off state of $v$ and its neighbors. Starting from an initial configuration of lights, one wins the game by finding a solution: a sequence of clicks that turns off all the lights. When $G$ is a $5 \times 5$ grid, this game was commercially available from Tiger Electronics as \textit{Lights Out}. Restricting ourselves to solvable initial configurations, we pose a natural question about this game, the Most Clicks Problem (MCP): How many clicks does a worst-case initial configuration
		on $G$ require to solve? The answer to the MCP is already known for nullity 0 graphs: those on which every initial configuration is solvable. Generalizing a technique from Scherphius, we give an upper bound to the MCP for all grids of size $(6k - 1) \times (6k - 1)$. We show the value given by this upper bound exactly solves the MCP for all nullity 2 grids of this size. We show that all nullity 2 grids are of size $(6k - 1) \times (6k - 1)$, which means we solve the MCP for all nullity 2 square grids.
	\end{abstract}
	
	\section{Introduction}
	Consider a game played on a simple graph $G=(V,E)$ where each vertex consists of a clickable light.
	Clicking any vertex $v$ toggles the on/off state of $v$ and its neighbors.
	Starting from an initial configuration of lights, one wins the game by finding a solution: a sequence of clicks that turns off all the lights.
	When $G$ is a $5 \times 5$ grid, this game was commercially available from Tiger Electronics as \textit{Lights Out}.
	
	Restricting ourselves to solvable initial configurations, we pose a natural question about this game, the \textit{Most Clicks Problem (MCP)}:
	\begin{quote}
		How many clicks does a worst-case initial configuration on $G$ require to solve?
	\end{quote}
	This question is analogous to finding the most turns of a Rubik's Cube needed to solve any scramble, which is known to be 20 \cite{Rockiki2013}.

	Sutner was one of the first to study these sorts of games, mostly through the lens of cellular automata.
	He observes that clicking a vertex twice is equivalent to not clicking it at all, and the order in which one clicks vertices is irrelevant.
	Thus, solutions can be sets of vertices rather than sequences \cite{Sutner1989}.
	Sutner also introduces the \textit{All Ones Problem}, which asks if it is possible for some graph $G$ to go from all lights off to all lights on \cite{Sutner1988}.
	He shows that it is always possible for any $G$ \cite{Sutner1989}.
	However, finding minimal such sets, the \textit{Minimum All Ones Problem}, is NP-complete \cite{Sutner1988}.
	Several authors have shown for different types of graphs whether or not the Minimum All Ones Problem remains NP-Complete.
	Chen, Li, Wang, and Zhang give linear time algorithms for trees, unicyclic and bicyclic graphs, and graphs with bounded tree width \cite{CHEN200493}.
	Amin and Slater give a linear time algorithm for series-parallel graphs \cite{AminSlater1992}.
	Broersma and Li show that the Minimum All Ones Problem remains NP-Complete when $G$ is restricted to be bipartite or planar \cite{BROERSMA200760}.
	
	Goldwasser and Klostermeyer show that the related problem of determining if some initial configuration can be transformed through light clicks to a configuration with at least $k$ lights off, which they call the \textit{Maximizing Off Switches (MOS)} problem, is NP-Complete \cite{Goldwasser2000}.
	However, they note that the particular case where $k = \abs{V}$, which is equivalent to asking if an initial configuration is solvable, has a polynomial time algorithm.
	Goldwasser and Klostermeyer do not believe MOS is NP-complete when restricted to grid graphs \cite{Goldwasser2000}.
	
	In analyzing these sorts of games, authors have found useful connections with other parts of mathematics.
	An \textit{odd dominating set} of $G$ is a subset $V' \subseteq V$ such that every vertex contains an odd number of vertices from $V'$ in its neighborhood (itself and its edge-adjacent neighbors).
	Even dominating sets can be analogously defined.
	Consider the function $\Phi_G$ from subsets of $V$ to subsets of $V$ that maps which vertices are clicked to which lights change state.
	Odd dominating sets are exactly the solutions to the All Ones Problem, while even dominating sets are exactly the elements of $\ker{\Phi_G}$ \cite{Sutner1989}.
	Sutner showed that this function is linear with respect to symmetric difference.
	That is, for all $V_1, V_2 \subseteq V$,
	\begin{equation*}
		\Phi_G(V_1 \triangle V_2) = \Phi_G(V_1) \triangle \Phi_G(V_2).
	\end{equation*}
	Further, the set of all configurations form a vector space, of which the set of all solvable configurations is a subspace \cite{Sutner1988}. 
	
	Finding $\dim{\ker{\Phi_G}}$, the \textit{nullity} of $G$, is useful for determining the difficulty of solving problems like MCP and MOS. 
	For example, every initial configuration is solvable in exactly one way if and only if the nullity of $G$ is 0, meaning the answer to MOS is ``yes'' for all $k \leq \abs{V}$.
	In general, only 1 out of every $2^{\dim{\ker{\Phi_G}}}$ initial configurations are solvable, and each solvable configuration can be solved in $2^{\dim{\ker{\Phi_G}}}$ unique ways \cite{Sutner1989}.
	There are other unexpected connections with the nullity of $G$.
	Eriksson, Eriksson, and Sjöstrand show that $G$ is nullity 0 if and only if the number of perfect matchings in $G$ is odd \cite{ERIKSSON2001357}.
	Sutner also shows that one can calculate the nullity of an $m \times n$ grid as the GCD of two Chebyshev polynomials in $\Z_2[x]$ \cite{Sutner2000}.
	
	The MCP is trivial when the nullity of $G$ is 0.
	The initial configuration obtained by starting with all lights off and clicking every vertex once can only be solved in the same way of clicking every vertex once.
	Thus, the answer to the MCP for nullity 0 graphs is $\abs{V}$.
	
	When $G$ is an $n \times n$ grid, the nullity of $G$ is always even \cite{Sutner1989}.
	Thus, the next nullity of interest is nullity 2.
	The smallest $n \times n$ grid graph with nullity 2 is $n=5$.
	Scherphuis showed that the answer to the MCP for the $5 \times 5$ grid graph is 15 \cite{jaap}.
	However, the answer to the MCP for all other grid sizes with non-zero nullity is unknown. 
	We generalize Scherphuis's technique to solve the MCP for a $5 \times 5$ grid to give an upper bound the MCP for all grids of size $(6k-1) \times (6k-1)$.
	We show the value given by this upper bound exactly solves the MCP for all nullity 2 grids of this size.
	We conjecture that all nullity 2 grids are of size $(6k-1) \times (6k-1)$, which would mean we solved the MCP for all nullity 2 grids.
	
	\section{Even Parity Covers}
	Scherphuis's techniques rely heavily on even parity covers.
	Thus, we need to first show basic properties of even parity covers.
	
	\begin{theorem}\label{iff-even-parity}
		The sets $V_1$ and $V_2$ are solutions to the same initial configuration if and only if $V_1 \triangle V_2$ is an even parity cover of $G$.
	\end{theorem}
	\begin{proof}
		Assume $V_1$ and $V_1$ are solutions to the same initial configuration.
		This means that they have the same image under $\Phi_G$.
		Notice,
		\begin{align*}
			\Phi_G(V_1 \triangle V_2) &= \Phi_G(V_1) \triangle \Phi_G(V_2) \\
				&= \Phi_G(V_1) \triangle \Phi_G(V_1) \\
				&= \emptyset.
		\end{align*}
		Thus, $V_1 \triangle V_2 \in \ker{\Phi_G}$ and is an even parity cover of $G$.
		Now assume that $V_1 \triangle V_2$ is an even parity cover of $G$.
		Notice,
		\begin{align*}
			\Phi_G(V_1) &= \Phi_G(V_1) \triangle \emptyset \\
				&= \Phi_G(V_1) \triangle \Phi_G(V_1 \triangle V_2) \\
				&= \Phi_G(V_1) \triangle \Phi_G(V_1) \triangle \Phi_G(V_2) \\
				&= \Phi_G(V_2).
		\end{align*}
		Thus, $V_1$ and $V_2$ are solutions to the same initial configuration.
	\end{proof}

	Theorem \ref{iff-even-parity} gives us a way to determine the solution to a given initial configuration that uses the fewest clicks.
	First, find all even parity covers of $G$ and any solution to the initial configuration.
	Next, take the symmetric difference of this solution with each even parity cover of $G$ to generate all solutions to the initial configuration.
	Finally, choose whichever solution is the smallest.

	\section{A Lower Bound on Nullity}
	For each even parity cover of a smaller grid, we can construct a unique even parity cover of a larger grid.
	This construction will allow us in the next section to generalize Scherphuis's proof technique to larger grids.
	
	\begin{theorem}\label{tiling-quiet-patterns}
		Let $d(n) = \dim{\ker{\Phi_G}}$ for $G$ an $n \times n$ grid.
		Then for all $n,k \in \N$,
		\begin{equation*}
			d(nk - 1) \geq d(n-1).
		\end{equation*}
	\end{theorem}
	\begin{proof}
		Since $nk - 1 = k(n-1) + (k-1)$, an $(nk-1) \times (nk-1)$ grid consists of a $k \times k$ grid of $(n-1) \times (n-1)$ grids, each separated by a horizontal or vertical strip of height or width 1.
	
		Let $Q$ be an even parity cover for an $(n-1) \times (n-1)$ grid.
		Let $R$ be $Q$ tiled $k$ times horizontally and vertically onto an $(nk-1) \times (nk-1)$ grid, where each tile is a horizontal/vertical reflection of its vertical/horizontal neighboring tiles, as shown below.
		
		\begin{figure}[H]
			\centering
% Code the generate figure w/o Q's
%			\begin{tikzpicture}[scale=0.33]
%				\edef\M{3} % Width of each tile
%				\edef\K{5} % Number of tiles in each direction
%				\pgfmathparse{(\M+1)*\K - 1} % Size of large grid
%				\edef\N{\pgfmathresult}
%				
%				\draw[black,thin] (0,0) grid (\N,\N);
%				
%				\foreach \i in {1,2,...,\K}{
%					\foreach \j in {1,2,...,\K}{
%						\pgfmathparse{(\i - 1) * (\M + 1)}
%						\edef\X{\pgfmathresult}
%						\pgfmathparse{(\j - 1) * (\M + 1)}
%						\edef\Y{\pgfmathresult}
%						
%						\draw[fill,gray,opacity=0.707] (\X, \Y) rectangle (\X + \M, \Y + \M);
%						%\node at (\X + \M/2, \Y + \M/2) {$Q$};
%					}
%				}
%			\end{tikzpicture}
			\includegraphics[width=0.5\textwidth]{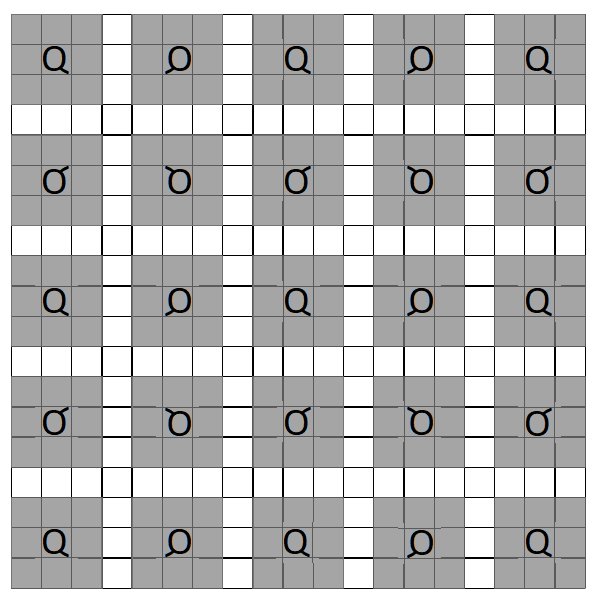}
			\caption{Pattern $Q$ tiled 5 times.}	
		\end{figure}
	
		Consider any vertex $v$ in the $(nk-1) \times (nk-1)$ grid.
		Let $N(v) = \{u \in V \mid (u,v) \in E \text{ or } u=v\}$.
		If $v$ is in one of the areas containing a reflection of $Q$, then $\abs{N(v) \cap R}$ is even because $Q$ is an even parity cover.
		Otherwise, $v$ is on the border between reflections of $Q$.
		If $v$ is in the intersection of a horizontal and vertical border, then $\abs{N(v) \cap R} = 0$, which is even.
		Otherwise, $N(v)$ contains two vertices that are in neighboring reflections of $Q$.
		Since we reflected $Q$ across this border, the two vertices are either both in $R$ or neither are in $R$.
		In either case, $\abs{N(v) \cap R}$ is even.
		So, $R$ is an even parity cover.
		
		So, each even parity cover of an $(n-1) \times (n-1)$ grid allows us to generate a unique even parity cover for an $(nk-1) \times (nk-1)$ grid.
		Since these even parity covers are independent on a $(n-1) \times (n-1)$ grid, they are also independent on the larger $(nk-1) \times (nk-1)$ grid.
		Thus,
		\begin{equation*}
			d(nk-1) \geq d(n-1).
		\end{equation*}
	\end{proof}

	Sutner shows a generalization of this result to product graphs \cite{Sutner1988_2}.

	\section{Solving the MCP}
	We are now ready to solve the MCP for all nullity 2 grids of size $(6k-1) \times (6k-1)$.
	Applying Theorem \ref{tiling-quiet-patterns} for $n=6$, we get that for all $k \in \N$,
	\begin{equation*}
		d(6k - 1) \geq d(5) = 2.
	\end{equation*}
	So, tiling the even parity covers from a $5 \times 5$ grid give even parity covers of $(6k-1) \times (6k-1)$ grids.
	For example, notice how the $17 \times 17$ even parity covers are simply the $5 \times 5$ even parity covers tiled $k=3$ times.
	
	\begin{figure}[H]
		\centering
		\includegraphics[width=0.32\textwidth]{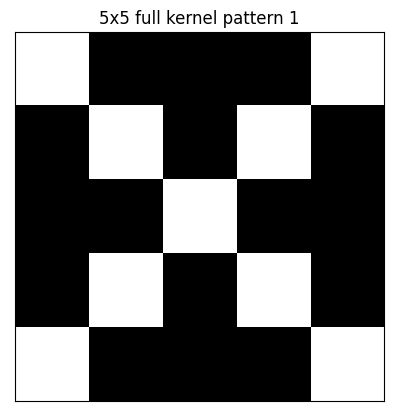}
		\includegraphics[width=0.32\textwidth]{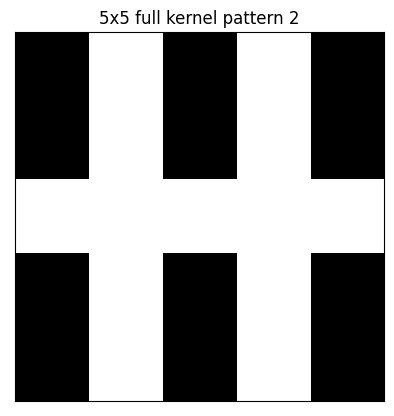}
		\includegraphics[width=0.32\textwidth]{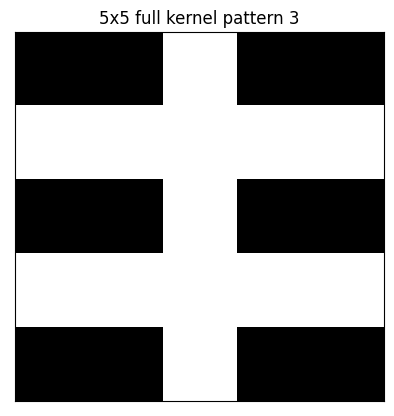}
		\includegraphics[width=0.32\textwidth]{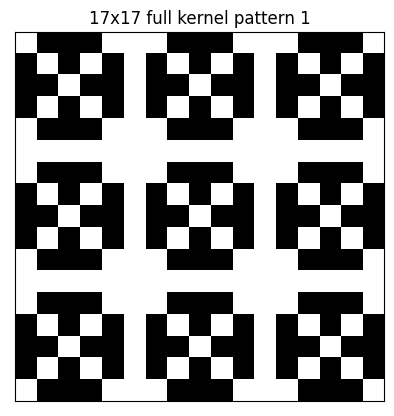}
		\includegraphics[width=0.32\textwidth]{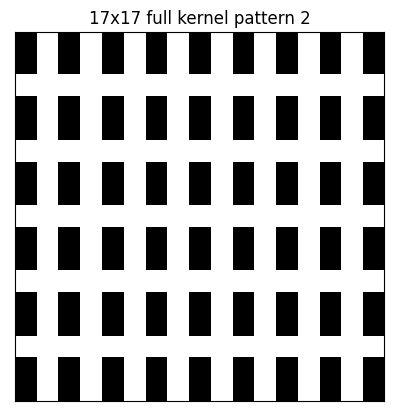}
		\includegraphics[width=0.32\textwidth]{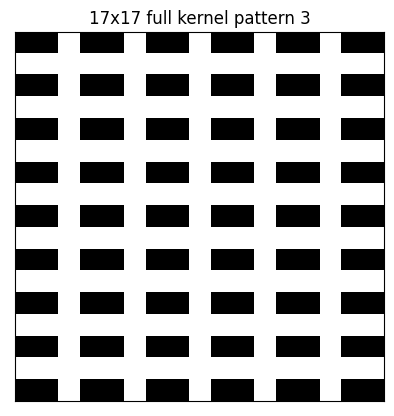}
		\caption{Non-trivial elements of $\ker{\Phi_G}$ for $G$ a $5 \times 5$ and $17 \times 17$ grid}
	\end{figure}
	
	We will apply the same techniques that Scherphius \cite{jaap} used to solve the MCP for a $5 \times 5$ board to all $(6k-1) \times (6k-1)$ boards.
	\begin{theorem}\label{min-moves-problem-6k-1x6k-1}
		Let $k \in \N$.
		Then the answer to the MCP for a $(6k - 1) \times (6k - 1)$ grid is at most $26k^2 - 12k + 1$.
		This bound is exactly the answer to the MCP when $d(6k-1)=2$.
	\end{theorem}
	\begin{proof}
		Applying Theorem \ref{tiling-quiet-patterns} with $n=6$, the even parity covers of the $5 \times 5$ grid will tile a $(6k - 1) \times (6k - 1)$ grid.
		We can partition the grid into four regions based on which of these even parity covers they are a part of.
		Below are the regions for the $5 \times 5$ board.
		For larger boards, these regions would be tiled exactly as described in Theorem \ref{tiling-quiet-patterns}.
		
		\begin{figure}[H]
			\centering
			\includegraphics[width=0.49\textwidth]{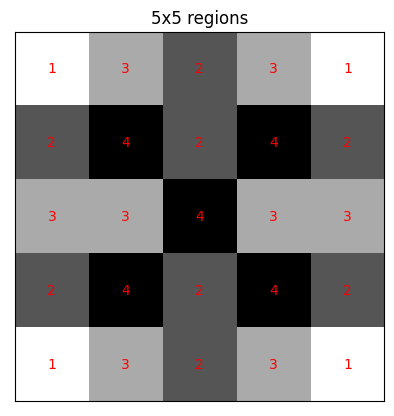}
			\caption{The four regions of a $5 \times 5$ grid}
		\end{figure}
		
		Region 1 is the intersection of even parity covers 2 and 3;
		region 2 is the intersection of even parity covers 1 and 2;
		region 3 is the intersection of even parity covers 1 and 3;
		region 4 is the intersection of the complements.
		Region 1 contains $4k^2$ vertices;
		regions 2 and 3 contain $8k^2$ vertices;
		region 4 contains the remaining $16k^2 - 12k + 1$ vertices.
		
		Let $Y$ be a worst-case initial configuration.
		Let $X$ be a solution to $Y$ that uses the fewest clicks among all solutions to $Y$.
		Let $R_i$ be the number of vertices in $X$ also contained in region $i$.
		Then the solution uses $\abs{X} = R_1 + R_2 + R_3 + R_4$ clicks.
		
		If we take the symmetric difference of $X$ and even parity cover 1, we obtain an equivalent solution that uses $R_1 + (8k^2 - R_2) + (8k^2 - R_3) + R_4$ clicks.
		Since we assumed $X$ was a minimal solution,
		\begin{equation*}
			R_1 + R_2 + R_3 + R_4 \leq R_1 + (8k^2 - R_2) + (8k^2 - R_3) + R_4.
		\end{equation*}
		Rearranging,
		\begin{equation*}
			R_2 + R_3 \leq 8k^2.
		\end{equation*}
		If we take the symmetric difference of $X$ and even parity cover 2, we obtain an equivalent solution that uses $(4k^2 - R_1) + (8k^2 - R_2) + R_3 + R_4$ clicks.
		Since we assumed $X$ was a minimal solution,
		\begin{equation*}
			R_1 + R_2 + R_3 + R_4 \leq (4k^2 - R_1) + (8k^2 - R_2) + R_3 + R_4.
		\end{equation*}
		Rearranging,
		\begin{equation*}
			R_1 + R_2 \leq 6k^2.
		\end{equation*}
		If we take the symmetric difference of $X$ and even parity cover 3, we obtain an equivalent solution that uses $(4k^2 - R_1) + R_2 + (8k^2 - R_3) + R_4$ clicks.
		Since we assumed $X$ was already a minimal solution,
		\begin{equation*}
			R_1 + R_2 + R_3 + R_4 \leq (4k^2 - R_1) + R_2 + (8k^2 - R_3) + R_4.
		\end{equation*}
		Rearranging,
		\begin{equation*}
			R_1 + R_3 \leq 6k^2.
		\end{equation*}
	
		We see that $R_4$ is not constrained by any of the above inequalities, only that region 4 contains $16k^2 - 12k + 1$ squares.
		Since we assumed $Y$ is a worst-case initial configuration, we can set $R_4 = 16k^2 - 12k + 1$.
		Meanwhile, $R_1$, $R_2$, and $R_3$ are constrained by the following integer linear program (ILP) where we seek to maximize $R_1 + R_2 + R_3$:
		\begin{equation*}
			\begin{bmatrix}
				0 & 1 & 1 \\
				1 & 1 & 0 \\
				1 & 0 & 1 
			\end{bmatrix}
			\begin{bmatrix}
				R_1 \\
				R_2 \\
				R_3
			\end{bmatrix}
			\leq
			\begin{bmatrix}
				8 \\
				6 \\
				6
			\end{bmatrix}k^2.
		\end{equation*}
		We see that
		\begin{equation*}
			\begin{bmatrix}
				R_1 \\
				R_2 \\
				R_3
			\end{bmatrix}
			=
			\begin{bmatrix}
				2 \\
				4 \\
				4
			\end{bmatrix}k^2
		\end{equation*}
		is the solution to the ILP.
		So,
		\begin{align*}
			\abs{X} &=  R_1 + R_2 + R_3 + R_4 \\
				&= 2k^2 + 4k^2 +  4k^2 + (16k^2 - 12k + 1) \\
				&= 26k^2 - 12k + 1.
		\end{align*}
		If $\dim{\ker{\Phi}} > 2$ there will be more even parity covers that would create more restrictive constraints.
		However, if $\dim{\ker{\Phi}} = 2$, we have considered all even parity covers and have solved the MCP.
	\end{proof}

	The $5 \times 5$, $17 \times 17$, $41 \times 41$, $53 \times 53$, $77 \times 77$, grids are those of size $(6k-1) \times (6k-1)$ less than 100 with nullity 2.
	Thus, we find that the answers to the MCP for these grids are 15, 199, 1191, 1999, and 4239 respectively.
	
	We will now show that all nullity 2 grids are of size $(6k-1) \times (6k-1)$, meaning the results of Theorem \ref{min-moves-problem-6k-1x6k-1} exactly solve the MCP for all nullity 2 grids.
	
	\begin{lemma}\label{all_nullity2_even}
		If $d(n) = 2$, then $n$ is odd.
	\end{lemma}
	\begin{proof}
		Assume for the sake of contradiction that $d(n) = 2$, but $n$ is even.
		Then $n = 2k$ for some natural number $k$.
		Then
		\begin{equation*}
			\deg{\gcd{\left(f_{2k+1}(x), f_{2k+1}(x+1)\right)}} = 2,
		\end{equation*}
		where $f_n(x)$ is the $n$th Fibonacci polynomial as defined by Boyles \cite{Boyles2022}.
		Since $2k+1$ is odd, $x \not\mid f_{2k+1}(x)$ and $x+1 \not\mid f_{2k+1}(x)$, as Boyles shows \cite{Boyles2022}.
		So, both $x$ and $x+1$ do not divide $\gcd{\left(f_{2k+1}(x), f_{2k+1}(x+1)\right)}$.
		Instead, the result of this GCD must be the degree 2 polynomial that is not divisible by $x$ or $x+1$.
		The only such polynomial is the irreducible polynomial $x^2 + x + 1$.
		
		However, as Hunziker, Machiavelo, and Park show, since $2k+1$ is odd, $f_{2k+1}(x)$ is the square of a square-free polynomial \cite{HUNZIKER2004465}.
		So, since $x^2 + x + 1$ is irreducible and $f_{2k+1}(x)$ is a square, we must have
		\begin{equation*}
			(x^2+x+1)^2 \mid f_{2k+1}(x).
		\end{equation*}
		Further, since $(x+1)^2 + (x+1) + 1 = x^2 + x + 1$, we also have 
		\begin{equation*}
			(x^2+x+1)^2 \mid f_{2k+1}(x+1).
		\end{equation*}
		Therefore, $(x^2+x+1)^2 \mid \gcd{\left(f_{2k+1}(x), f_{2k+1}(x+1)\right)}$, meaning $d(n) \geq 4$, a contradiction.
	\end{proof}

	\begin{theorem}\label{nullity2-5mod6}
		If $d(n) = 2$, then $n \equiv -1 \mod 6$.
	\end{theorem}
	\begin{proof}
		Let $n$ bw a natural number such that $d(n) = 2$.
		Let $n = b\cdot2^k - 1$ for some odd natural number $b$ and natural number $k$.
		As Boyles shows, there are 2 cases \cite{Boyles2022},
		\begin{enumerate}
			\item
				$d(b-1) = 0$ and $b \equiv 0 \mod 3$, meaning $d(n) = d(2b-1) = 2$.
			\item
				$d(b-1) = d(n) = 2$.
		\end{enumerate}
		As we've shown in Lemma \ref{all_nullity2_even}, case 2 is not possible.
		So, $k = 1$ and $b$ is a multiple of 3.
		Let $b = 3m$ for some natural number $m$.
		So, $n = 2b - 1 = 2(3m) - 1 = 6m - 1 \equiv -1 \mod 6$, as desired. 
	\end{proof}

	Together, Theorems \ref{min-moves-problem-6k-1x6k-1} and \ref{nullity2-5mod6} mean that we have exactly solved the MCP for all square nullity 2 boards.

	\begin{corollary}\label{nullity2-5mod12}
		If $d(n) = 2$, then $n \equiv -7 \mod 12$.
	\end{corollary}
	\begin{proof}
		Let $n$ be a natural number such that $d(n) = 2$.
		We showed in the proof to Theorem \ref{nullity2-5mod6} that $n = 2b - 1$, where $b$ is an odd multiple of 3.
		Thus, $b = 6k - 3$ for some natural number $k$.
		So, $n = 2(6k - 3) - 1 = 12k - 7 \equiv -7 \mod 12$, as desired.
	\end{proof}

	\section{Conjectures \& Future Work}
	\begin{conjecture}\label{infinite-nullity-2}
		There are infinitely many $n$ such that $d(n) = 2$.
	\end{conjecture}
	The version of this conjecture with $d(n) = 0$ was proven by Sutner, who showed that $d(2^k - 1) = 0$ \cite{Sutner1989}.
	We suspect that $d(2 \cdot 3^k - 1) = 2$.
	Theorem \ref{tiling-quiet-patterns} tells us that if there exists one grid with nullity $q$, there are infinitely many $n$ such that $d(n) \geq q$.
	So, there are certainly infinitely many grids with nullity at least 2.
	By computer search, we found for $n \leq 25000$ there are 1242 nullity 2 grids, about 5\% of all $n$.
	This is well below our best estimate of about 1 in 12  boards being nullity 2 as given by Corollary \ref{nullity2-5mod12}.
	
	We'd like future work to focus on resolving these conjectures and solving the MCP on grids with nullities greater than 2.
	For any sized grid (in fact, any graph), we can apply the same technique as in Theorem \ref{min-moves-problem-6k-1x6k-1} of finding the answer to the MCP as the solution to an ILP.
	
	We'd also like future work to focus more on understanding the relationships between the nullities of different sized grids.
	Theorem \ref{tiling-quiet-patterns} is a good start to this, but there seem to be many other relationships to discover.
	For example, Knuth shows that if an even parity cover of an $n \times n$ grid is perfect -- every row and column contains elements from the cover -- then we can construct a perfect even parity cover for a $(2n+1) \times (2n+1)$ grid \cite{Knuth_AOCP4A}.
	All the even parity covers constructed in Theorem \ref{tiling-quiet-patterns} are not perfect.
	
	\newpage
	\bibliography{refs.bib}
	\bibliographystyle{amsplain}
\end{document}